\documentclass[11pt,reqno]{amsart}
\usepackage{amsfonts,amsthm,amssymb}
\usepackage{amsmath}
\usepackage{pdfsync}

\usepackage[mathscr]{eucal}
\usepackage{graphicx}

\newtheorem{thm}{Theorem}

\newtheorem{lemma}{Lemma}

\newtheorem{remark}{Remark}[section]

\newcommand{\R}{\ensuremath{\mathbb{R}}}

\newcommand{\N}{\ensuremath{\mathbb{N}}}

\newcommand{\F}{\ensuremath{\mathcal{F}}}

\newcommand{\wh}{\ensuremath{\widehat}}

\newcommand{\1}{\ensuremath{\mathbf{1}}}

\newcommand{\bx}{\ensuremath{\mathbf{x}}}

\newcommand{\eps}{\ensuremath{\varepsilon}}
\newcommand{\la}{\ensuremath{\lambda}}

\newcommand{\al}{\ensuremath{\alpha}}

\newcommand{\Ga}{\ensuremath{\Gamma}}

\newcommand{\de}{\ensuremath{\delta}}
\newcommand{\De}{\ensuremath{\Delta}}
\newcommand{\si}{\ensuremath{\sigma}}
\newcommand{\om}{\ensuremath{\omega}}

\newcommand{\tmu}{\ensuremath{\tilde{\mu}}}

\newcommand{\ls}{\ensuremath{\lesssim}}

\newcommand{\subs}{\ensuremath{\subseteq}}

\numberwithin{equation}{section}

\newcommand{\eq}{\begin{equation}
\newcommand{\ee}{\end{equation}}}

\title{Simplices in thin subsets of Euclidean spaces}

\author{Alex Iosevich and \'Akos Magyar }

\thanks{The research of the first listed author was partially supported by the National Science Foundation grant no. HDR TRIPODS - 1934962. The research of the second listed author was partially supported by the National Science Foundation grant NSF-DMS 1600840.}

\begin{document}

\begin{abstract} Let $\De$ be a non-degenerate simplex on $k$ vertices. We prove that there exists a threshold $s_k<k$ such that any set $A\subs \R^k$ of Hausdorff dimension $dim\,A\geq s_k$ necessarily contains a similar copy of the simplex $\De$.  
\end{abstract}

\maketitle

\section{Introduction.}

A classical problem of geometric Ramsey theory is to show that a sufficiently large sets contain a given geometric configuration. The underlying settings can be the Euclidean space, the integer lattice or vector spaces over finite fields. By a geometric configuration we understand the collection of finite point sets obtained from a given finite set $F\subs\R^k$ via translations, rotations and dilations.

\vskip.125in 

If the size is measured in terms of the positivity of the Lebesgue density, then it is known that large sets in $\R^k$ contain a translated and rotated copy of all sufficiently large dilates of any non-degenerate simplex $\De$ with $k$ vertices \cite{Bo3}. However, on the scale of the Hausdorff dimension $s<k$ this question is not very well understood, the only affirmative result in this direction obtained by Iosevich-Liu \cite{IL07}.

\vskip.125in 

In the other direction, a construction due to Keleti \cite{Keleti08} shows that there exists set $A\subs\R$ of full Hausdorff dimension which do not contain any non-trivial 3-term arithmetic progression. In two dimensions an example due to Falconer \cite{Falconer} and Maga \cite{Maga} shows that there exists set $A\subs \R^2$ of Hausdorff dimension 2, which do not contain the vertices of an equilateral triangle, or more generally a non-trivial similar copy of a given non-degenerate triangle. It seems plausible that examples of such sets exist in all dimensions, but this is not currently known. See (\cite{FP11}) for related results. 

\vskip.125in 

The purpose of this paper is to show that measurable sets $A\subs \R^k$ of sufficiently large Hausdorff dimension $s<k$ contain a similar copy of any given non-degenerate $k$-simplex with bounded eccentricity. Our arguments make use of and have some similarity to those of Lyall-Magyar \cite{LM18graph}. We also extend out results to bounded degree distance graphs. For the special case of a path (or chain), and, more generally, a tree, similar but somewhat stronger results were obtained in \cite{BIT16} and \cite{IT19}.

\section{Main results.}

Let $V=\{v_1,\dots,v_k\}\subs\R^k$ be a non-degenerate $k$-simplex, a set of $k$ vertices which are in \emph{general position} spanning a $k-1$-dimensional affine subspace. For $1\leq j\leq k$ let $r_j(V)$ be the distance of the vertex $v_j$ to the affine subspace spanned by the remaining vertices $v_i,\ i\neq j$ and define $r(V):=\min_{1\leq j\leq k} r_j(V)$. Let $d(V)$ denote the diameter of the simplex, which is also the maximum distance between two vertices. Then the quantity $\de(V):=r(V)/d(V)$, which is positive if and only if $V$ is non-degenerate, measures how close the simplex $V$ is to being degenerate.

We say that a simplex $V'$ is \emph{similar} to $V$, if $V'=x+\la\cdot U(V)$ for some $x\in\R^k$, $\la>0$ and $U\in SO(k)$, that is if $V'$ is obtained from $V$ by a translation, dilation and rotation.

\begin{thm}\label{Thm2.1} Let $k\in\N$, $\de>0$. There exists $s_0=s_0(k,\de) < k$ such that if $E$ is a compact subset of $\R^k$ of Hausdorff dimension $dim\,E \geq s_0$, then $E$ contains the vertices of a simplex $V'$ similar to $V$, for any non-degenerate $k$-simplex $V$ with $\de(V)\geq \de$.
\end{thm}

\begin{remark} Note that the dimension condition is sharp for $k=2$ as a construction due to Maga \cite{Maga} shows the existence of a set $E\subs \R^2$ with $dim (E)=2$ which does not contain any equilateral triangle or more generally a similar copy of any given triangle. \end{remark} 

\begin{remark} It is also interesting to note that the proof of Theorem \ref{Thm2.1} above proves much more than just the existence of vertices of $V'$ similar to $V$ inside $E$. The proof proceeds by constructing a natural measure on the set of simplexes and proving an upper and a lower bound on this measure. This argument shows that an infinite "statistically" correct "amount" of simplexes $V'$s that satisfy the conclusion of the theorem exist, shedding considerable light on the structure of set of positive upper Lebesgue density. 
\end{remark} 

\begin{remark} Theorem \ref{Thm2.1} establishes a non-trivial exponent $s_0<k$, but the proof yields $s_0$ very close to $k$ and not explicitly computable. The analogous results in the finite field setting (see e.g. \cite{HI08}, \cite{IP19} and the references contained therein) suggest that it may be possible to obtain explicit exponents, but this would require a fundamentally different approach to certain lower bounds obtained in the proof of Theorem \ref{Thm2.1}. \end{remark} 

\vskip.125in 


A distance graph is a connected finite graph embedded in Euclidean space, with a set of vertices $V=\{v_0,v_1,\ldots,v_n\}\subs\R^d$ and a set of edges $E\subs \{(i,j);\ 0\leq i<j\leq n\}$. We say that a graph $\Ga=(V,E)$ has degree at most $k$ if $|V_j|\leq k$ for all $1\leq j\leq n$, where $V_j=|\{v_i:\ (i,j)\in E\}|$. The graph $\Ga$ is called \emph{proper} if the sets $V_j\cup \{v_j\}$ are in general position. Let $r(\Ga)$ be the minimum of the distances from the vertices $v_j$ to the corresponding affine subspace spanned by the sets $V_j$ and note that $r(\Ga)>0$ if $\Ga$ is proper. Let $d(\Ga)$ denote length of the longest edge of $\Ga$ and let $\de(\Ga):=r(\Ga)/d(\Ga)$.\\

We say that a distance graph $\Ga'=(V',E)$ is \emph{isometric} to $\Ga$, and write $\Ga'\simeq \Ga$ if there is a one-one and onto mapping $\phi:V\to V'$ so that $|\phi(v_i)-\phi(v_j)|=|v_i-v_j|$ for all $(i,j)\in E$. One may picture $\Ga'$ obtained from $\Ga$ by a translation followed by rotating the edges around the vertices, if possible. By $\la\cdot \Ga$ we mean the dilate of the distance graph $\Ga$ by a factor $\la>0$ and we say that $\Ga'$ is \emph{similar} to $\Ga$ if $\,\Ga'$ is isometric to $\la\cdot\Ga$. 

\begin{thm}\label{Thm2.2} Let $\de>0$, $n\geq 1$, $1\leq k<d$ and let $E$ be a compact subset of $\R^k$ of Hausdorff dimension $s<d$. There exists $s_0=s_0(n,d,\de)<d$ such if $s\geq s_0$ then $E$ contains a distance graph $\Ga'$ similar to $\Ga$, for any proper distance graph $\Ga=(V,E)$ of degree at most $k$, with $V\subs\R^d$, $|V|=n$ and $\de(\Ga)\geq\de$.
\end{thm}

Note that Theorem \ref{Thm2.2} implies Theorem \ref{Thm2.1} as a non-degenerate simplex is a proper distance graph of degree  $k-1$.

\section{Proof of Theorem \ref{Thm2.1}.}

Let $E\subs B(0,1)$ be a compact subset of the unit ball $B(0,1)$ in $\R^k$ of Hausdorff dimension $s<k$. It is well-known that there is a probability measure $\mu$ supported on $E$ such that $\mu(B(x,r))\leq C_\mu r^s$ for all balls $B(x,r)$. The following observation shows that we may take $C_\mu =4$ for our purposes. \footnote{We'd like to thank Giorgis Petridis for bringing this observation to our attention.}

\begin{lemma}\label{L3.1} There exists a set $E'\subs B(0,1)$ of the form $E'=\rho^{-1}(F-u)$ for some $\rho>0$, $u\in \R^k$ and $F\subs E$, and a probability measure $\mu'$ supported on $E'$ which satisfies
\eq\label{3.1}
\mu'(B(x,r)\leq 4 r^s,\quad\textit{for all}\quad x\in\R^k,\ r>0.\ee
\end{lemma}

\begin{proof} Let $K:= \inf (S)$, where
\[S:=\{C\in\R:\ \mu(B(x,r))\leq C r^s,\ \ \forall\ B(x,r)\}.\]
By Frostman's lemma \cite{Mattila} we have that $S\neq\emptyset$, $K>0$, moreover 
\[\mu(B(x,r))\leq 2K\,r^s,\]
for all balls $B(x,r)$. There exists a ball $Q = B(v,\rho)$ or radius $\rho$ such that $\mu(Q) \geq \frac{1}{2} K\rho^s$. We translate $E$ so $Q$ is centered at the origin, set $F = E \cap Q$ and denote by $\mu_F$ the induced probability measure on $F$
\[\mu_F(A) = \frac{\mu(A\cap F)}{\mu(F)}.\]
Note that for all balls $B=B(x,r)$,
\[\mu_F(B) \leq \frac{2K\,r^s}{\frac{1}{2} K\rho^s} =4 \left(\frac{r}{\rho}\right)^s.\]
Finally we define the probability measure $\mu'$, by $\,\mu'(A) := \mu_F(\rho A)$. It is supported on $E'=\rho^{-1} F\subs B(0,1)$ and satisfies 
\[\mu'(B(x,r)) = \mu_F(B(\rho x,\rho r)) \leq 4r^s.\]
\end{proof}

Clearly $E$ contains a similar copy of $V$ if the same holds for $E'$, thus one can pass from $E$ to $E'$ and hence assuming that \eqref{3.1} holds, in proving our main results. Given $\eps>0$ let $\psi_\eps(x)=\eps^{-k}\psi(x/\eps)\geq 0$, where $\psi\geq 0$ is a  Schwarz function whose Fourier transform,  $\widehat{\psi}$, is a compactly supported smooth function, satisfying $\widehat{\psi}(0)=1$ and $0\leq \widehat{\psi}\leq 1$.

We define $\mu_\eps:=\mu\ast \psi_\eps$. Note that $\mu_\eps$ is a continuous function satisfying $\|\mu_\eps\|_\infty \leq C \eps^{s-k}$ with an absolute constant $C=C_\psi >0$, by Lemma \ref{L3.1}.\\

Let $V=\{v_0=0,\ldots,v_{k-1}\}$ be a given a non-degenerate simplex and note that in proving Theorem \ref{Thm2.1} we may assume that $d(V)=1$ hence $\de(V)=r(V)$. A simplex $V'=\{x_0=0,x_1,\ldots,x_{k-1}\}$ is isometric to $V$ if for every $1\leq j\leq k$ one has that $x_j\in S_{x_1,\ldots,x_{j-1}}$, where 
\[S_{x_1,\ldots,x_{j-1}} =\{y\in\R^k:\ |y-x_i|=|v_j-v_i|,\ \ 0\leq i<j\}\]
is a sphere of dimension $k-j$, of radius $r_j=r_j(V)\geq r(V)>0$. Let $\si_{x_1,\ldots,x_{j-1}}$ denote its normalized surface area measure.\\ 

Given $0<\la,\eps\leq 1$ define the multi-linear expression,

\begin{align}\label{3.2}
&T_{\la V}(\mu_\eps)\ :=\\
&\int \mu_\eps(x) \mu_\eps(x-\la x_1)\cdots \mu_\eps(x-\la x_{k-1})\, d\si(x_1)\,d\si_{x_1}(x_2)\ldots d\si_{x_1,\ldots,x_{k-2}}(x_{k-1})\,dx,\nonumber
\end{align}
which may be viewed as a weighted count of the isometric copies of $\la \De$.


We have the following crucial upper bound

\begin{lemma}\label{L3.2} There exists a constant $C_k>0$, depending only on $k$, such that
\eq\label{3.3}
|T_{\la V} (\mu_{2\eps}) - T_{\la V}(\mu_\eps)| \leq\, C_k\,r(V)^{-\frac{1}{2}}\ \la ^{-\frac{1}{2}} \eps^{(k-\frac{1}{2})(s-k)+\frac{1}{4}}.
\ee
\end{lemma}

As an immediate corollary we have that 

\begin{lemma}\label{L3.3} Let $k-\frac{1}{4k}\leq s<k$. There exists 
\eq\label{3.4}
T_{\la V} (\mu) := \lim_{\eps\to 0}  T_{\la V} (\mu_\eps),
\ee
moreover 
\eq\label{3.5}
|T_{\la V} (\mu) - T_{\la V} (\mu_\eps)| \leq\,C_k\,r(V)^{-\frac{1}{2}}\ \la ^{-\frac{1}{2}} \eps^{(k-\frac{1}{2})(s-k)+\frac{1}{4}}. 
\ee
\end{lemma}

Indeed, the left side of \eqref{3.5} can be written as telescopic sum: 
\[\sum_{j\geq 0} T_{\la V}(\mu_{2\eps_j})-T_{\la V}(\mu_{\eps_j})\quad  \textit{with}\quad \eps_j = 2^{-j}\eps.\]

\begin{proof}[Proof of Lemma \ref{L3.2}] 

Write $\De\mu_\eps := \mu_{2\eps}-\mu_\eps$, then
\[\prod_{j=1}^{k-1} \mu_{2\eps}(x-\la x_j) - \prod_{j=1}^{k-1} \mu_{\eps}(x-\la x_j)=\sum_{j=1}^k \De_j(\mu_\eps),
\]
where
\eq\label{3.6}\De_j (\mu_\eps) = \prod_{i\neq j}\mu_{\eps_{ij}}(x-\la x_i)\,\De\mu_\eps (x-\la x_j),\ee
where $\eps_{ij}=2\eps$ for $i<j$ and $\eps_{ij}=\eps$ for $i>j$. Since the arguments below are the same for all $1\leq j\leq k-1$, assume $j=k-1$ for simplicity of notations. Writing $f\ast_\la g (x) :=\int f(x-\la y)g(y)\,dy$, and using 
$\|\mu_\eps\|_\infty \leq C \eps^{s-k}$, we have for $\De T(\mu_\eps):=T_{\la V}(\mu_\eps)-T_{\la V}(\mu_{2\eps})$,
\eq\label{3.7}|\De T (\mu_\eps)| \ls\,\eps^{(k-2)(s-d)}\,
\int \left|\int \mu_\eps(x)\ \De\mu_\eps\ast_\la \si_{x_1,\ldots,x_{k-2}}(x)\,dx\right|\,d\om(x_1,\ldots,x_{k-2})
\ee
where $d\om(x_1,\ldots,x_{k-2})=d\si(x_1)\ldots d\si_{x_1,\ldots,x_{k-3}}(x_{k-2})$ for $k>3$, while  for $k=3$ we have that $d\om(x_1)=d\si(x_1)$ the normalised surface area measure on the sphere $S=\{y:\ |y|=|v_1|\}$. 

 The inner integral is of the form 
\[|\langle\mu_\eps,\De_\eps\mu\ast_\la \si_{x_1,\ldots,x_{k-2}}\rangle| \ls \eps^{s-d}\, \|\De\mu_\eps\ast_\la \si_{x_1,\ldots,x_{k-2}}\|_2,\]

thus by Cauchy-Schwarz and Placherel's identity 
\[|\De_{k-1} T (\mu_\eps)|^2 \ls\,\eps^{2(k-1)(s-d)}\,\int |\widehat{\De\mu_\eps}(\xi)|^2\,I_\la (\xi)\,d\xi,\]
where
\[I_\la (\xi) = \int |\hat{\si}_{x_1,\ldots,x_{k-2}}(\la\xi)|^2\,d\om (x_1,\ldots,x_{k-2}).
\]
Since $S_{x_1,\ldots,x_{k-2}}$ is a 1-dimensional circle of radius $r_{k-1}\geq r(V)>0$, contained in an affine subspace orthogonal to $M_{x_1,\ldots,x_{k-2}}=Span\{x_1,\ldots,x_{k-2}\}$, we have that
\[|\hat{\si}_{x_1,\ldots,x_{k-2}}(\la\xi)|^2 \ls (1+r(V)\la\ dist (\xi,M_{x_1,\ldots,x_{k-2}}))^{-1}.\]

Since the measure $\om(x_1,\ldots,x_{k-2})$ is invariant with respect to that change of variables $(x_1,\ldots,x_{k-2})\to (Ux_1,\ldots,Ux_{k-2})$ for any rotation $U\in SO(k)$, one estimates
\begin{align*}
I_\la (\xi) &\ls \int \int (1+r(V)\la\ dist (\xi,M_{Ux_1,\ldots,Ux_{k-2}}))^{-1}\,d\om (x_1,\ldots,x_{k-2})\,dU\\
&= \int \int (1+r(V)\la\ dist (U\xi,M_{x_1,\ldots,x_{k-2}}))^{-1}\,d\om (x_1,\ldots,x_{k-2})\,dU\\
&= \int \int (1+r(V)\la\,|\xi|\ dist (\eta,M_{x_1,\ldots,x_{k-2}}))^{-1}\,d\om (x_1,\ldots,x_{k-2})\,d\si_{k-2}(\eta)\\
&\ls (1+r(V)\, \la\, |\xi|)^{-1},
\end{align*}
where we have written $\eta := |\xi|^{-1} U\xi$ and $\si_{k-1}$ denotes the surface area measure on the unit sphere $S^{k-1}\subs \R^k$.\\


Note that $\widehat{\De\mu_\eps}(\xi)=\hat{\mu}(\xi)(\hat{\psi}(2\eps\xi)-\hat{\psi}(\eps\xi))$, which is supported on $|\xi|\ls \eps^{-1}$ and is essentially supported on $|\xi|\approx \eps^{-1}$. Indeed, writing
\begin{align*}
J:&=\int |\widehat{\De\mu_\eps}(\xi)|^2\,I_\la (\xi)\,d\xi\\
  &= \int_{|\xi|\leq \eps^{-1/2}} |\widehat{\De\mu_\eps}(\xi)|^2\,I_\la (\xi)\,d\xi + \int_{\eps^{-1/2}\leq |\xi|\ls \eps^{-1}} |\widehat{\De\mu_\eps}(\xi)|^2\,I_\la (\xi)\,d\xi =: J_1+J_2.
  \end{align*}
Using  $|\hat{\psi}(2\eps\xi)-\hat{\psi}(\eps\xi)|\ls \eps^{1/2}$ for $|\xi|\leq \eps^{-1/2}$, we estimate
\[ J_1\, \ls\, \eps^\frac{1}{2}\,\int |\widehat{\mu}(\xi)|^2\,(\hat{\psi}(2\eps\xi)+\hat{\psi}(\eps\xi))\,d\xi\,\ls\,
\eps^{\frac{1}{2}+s-k},\]
as
\[\int |\hat{\mu}(\xi)|^2 \hat{\psi}(\eps\xi)\,d\xi\,=\,\int \mu_\eps(x)\,d\mu(x) \ls \eps^{s-k}.\]
On the other hand, as $I_\la(\xi)\ls \eps^{1/2}r(V)^{-1}\la^{-1}$ for $|\xi|\geq \eps^{-1/2}$ we have
\[J_2\,\ls\,\eps^{1/2}r(V)^{-1}\la^{-1}\, \int |\hat{\mu}(\xi)|^2 \hat{\phi}(\eps\xi)\,d\xi\,\ls\,
r(V)^{-1}\la^{-1}\eps^{\frac{1}{2}+s-k},\]
where we have written $\hat{\phi}(\xi)= (\hat{\psi}(2\xi)-\hat{\psi}(\xi))^2$. Plugging this estimates into \eqref{3.7} we obtain
\[|\De T (\mu_\eps)|^2\,\ls\,r(V)^{-1}\la^{-1}\eps^{\frac{1}{2}+(2k-1)(s-d)},\]
and \eqref{3.5} follows.
\end{proof}



The support of $\mu_\eps$ is not compact, however as it is a rapidly decreasing function it can be made to be supported in small neighborhood of the support of $\mu$ without changing our main estimates. Let $\phi_\eps(x):=\phi(c\,\eps^{-1/2}x)$ with some small absolute constant $c>0$, where $0\leq \phi(x)\leq 1$ is a smooth cut-off, which equals to one for $|x|\leq 1/2$ and is zero for $|x|\geq 2$. Define $\tilde{\psi}_\eps=\psi_\eps\,\phi_\eps$ and $\tmu_\eps = \mu\ast\tilde{\psi}_\eps$. It is easy to see that $\tmu_\eps\leq \mu_\eps$ and $\int\tmu_\eps\,\geq 1/2$, if $c>0$ is chosen sufficiently small.
Using the trivial upper bound, for $k-\frac{1}{4k}\leq s<k$ we have
\[|T_{\la\De}(\mu_\eps)-T_{\la\De}(\tmu_\eps)|\leq C_k\, \|\mu_\eps\|_\infty^{k-1}\, \|\mu_\eps - \tmu_\eps\|_\infty \leq C_k\,\eps^{1/2},\]
it follows that estimate \eqref{3.5} remains true with $\mu_\eps$ replaced with $\tmu_\eps$.\\


Let $f_\eps := c\,\eps^{k-s}\tmu_\eps$, where $c=c_\psi >0$ is a constant so that $0\leq f_\eps \leq 1$ and $\int f_\eps\,dx = c'\,\eps^{k-s}$. Let $\al:= c'\,\eps^{k-s}$ and note that the set $A_\eps := \{x:\ f_\eps (x)\geq \al/2\}$ has measure $|A_\eps|\geq \al/2$. We apply Theorem 2 (ii) together with the more precise lower bound (18) in \cite{LM18graph} for the set $A_\eps$. 

This gives that there exists and interval $I$ of length
$\,|I| \geq \exp\,(-\eps^{-C_k(d-s)})\,$, such that for all $\la\in I$, one has
$\ |T_{\la V}(A_\eps))| \geq\,c\,\al^{k}= c\,\eps^{k(k-s)}\ $, where
\begin{multline*}T_{\la V}(A_\eps) = \\
\int \1_{A_\eps}(x)\1_{A_\eps}(x-\la x_1)\ldots\1_{A_\eps}(x-\la x_{k-1})\,d\si(x_1)\ldots d\si_{x_1,\ldots,x_{k-2}}(x_{k-1})\,dx.
\end{multline*}
Since 
\[T_{\la\De}(\tmu_\eps)\geq 
c\,\al^{k} T_{\la v}(A_\eps),\] 
we have that 
\eq\label{3.6}
T_{\la V}(\tmu_\eps) \geq c>0,
\ee
\\
for all $\la\in I$, for a constant $c=c(k,\psi,r(V))>0$.\\

Now, let 
\[T_V(\tmu_\eps) := \int_0^1 \la^{1/2}\,T_{\la V}(\tmu_\eps)\,d\la.\]
For $k-\frac{1}{4k}\leq s <k$, by \eqref{3.5} we have that 
\[|T_{\la V}(\mu)- T_{\la V}(\tmu_\eps)| \leq \, C_k\,r(V)^{-\frac{1}{2}}\, 
\la ^{-\frac{1}{2}}\, \eps^\frac{1}{8},\]
it follows that 
\eq\label{3.7}
\int_0^1 \la^{1/2}\,|T_{\la V}(\mu)-T_{\la V}(\tmu_\eps)|\,d\la \leq\,C_k\,r(V)^{-\frac{1}{2}}\, \eps^\frac{1}{8},
\ee
and in particular $\ \int_0^1 \la^{1/2}\,T_{\la V}(\mu)\,d\la <\infty$. On the other hand by \eqref{3.6}, one has  
\eq\label{3.8}
\int_0^1 \la^{1/2}\,T_{\la V}(\tmu_\eps)\,d\la \geq  \exp\,(-\eps^{-C_k(k-s)}).
\ee
Assume  that $r(V)\geq \de$, fix a small $\eps=\eps_{k,\de}>0$ and the choose $s=s(\eps,\de)<k$ such that 
\[C_k\,\de^{-\frac{1}{2}}\, \eps^\frac{1}{8}<\frac{1}{2}\, \exp\,(-\eps^{-C_k(k-s)}),\]
which ensures that 
\[\int_0^1 \la^{1/2}\,T_{\la V}(\mu)\,d\la >0,\]
thus there exist $\la>0$ such that $T_{\la V}(\mu)>0$.
Fix such a $\la$, and assume indirectly that  $E^k=E\times\ldots\times E$ does not contain any simplex isometric to $\la V$, i.e. any point of the compact configuration space $S_{\la V}\subs \R^{k^2}$ of such simplices. By compactness, this implies that there is some $\eta>0$ such that the $\eta$-neighborhood of $E^k$ also does not contain any simplex isometric to $\la V$. As the support of $\tmu_\eps$ is contained in the $C_k\eps^{1/2}$-neighborhood of $E$, as $E= supp\,\mu$, it follows that $T_{\la V}(\tmu_\eps)=0$ for all $\eps<c_k\,\eta^2$ and hence $T_{\la V}(\mu)=0$, 
contradicting our choice of $\la$. This proves Theorem \ref{Thm2.1}.




\section{The configuration space of isometric distance graphs.}

Let $\Ga_0=(V_0,E)$ be a fixed proper distance graph, with vertex set\\ $V_0=\{v_0=0,v_1,\ldots,v_n\}\subs\R^d$ of degree $k<d$. Let $t_{ij}=|v_i-v_j|^2$ for $(i,j)\in E$. A distance graph $\Ga=(V,E)$ with $V=\{x_0=0,x_1,\ldots,x_n\}$ is isometric to $\Ga_0$ if and only if $\bx=(x_1,\ldots,x_n)\in S_{\Ga_0}$, where
\[
S_{\Ga_0}= \{(x_1,\ldots,x_n)\in\R^{dn};\ |x_i-x_j|^2=t_{ij},\ \forall\ \ 0\leq i<j\leq n,\ (i,j)\in E\}
\]

We call the algebraic set $S_{\Ga_0}$ the \emph{configuration space} of isometric copies of the $\Ga_0$. Note that $S_{\Ga_0}$ is the zero set of the family $\F=\{f_{ij};\ (i,j)\in E\}$, $f_{ij}(\bx)=|x_i-x_j|^2-t_{ij}$, thus it is a special case of the general situation described in Section 5.\\

If $\Ga\simeq \Ga_0$ with vertex set $V=\{x_0=0,x_1,\ldots,x_n\}$ is proper then $\bx=(x_1,\ldots,x_n)$ is a non-singular point of $S_{\Ga_0}$. Indeed, for a fixed $1\leq j\leq n$ let $\Ga_j$ be the distance graph obtained from $\Ga$ by removing the vertex $x_j$ together with all edges emanating from it. By induction we may assume that $\bx'=(x_1,\ldots,x_{j-1},x_{j+1},\ldots,x_n)$ is a non-singular point i.e the gradient vectors $\nabla_{\bx'}f_{ik}(\bx)$, $(i,k)\in E$, $i\neq j,k\neq j$ are linearly independent. Since $\Ga$ is proper the gradient vectors $\nabla_{x_j}f_{ij}(\bx)=2(x_i-x_j)$, $(i,j)\in E$ are also linearly independent hence $\bx$ is a non-singular point. In fact we have shown that the partition of coordinates $\bx=(y,z)$ with $y=x_j$ and $z=\bx'$ is admissible and hence \eqref{A7} holds.\\

Let $r_0=r(\Ga_0)>0$. It is clear that if $\Ga\simeq\Ga_0$ and $|x_j-v_j|\leq \eta_0$ for all $1\leq j\leq n$, for a sufficiently small $\eta=\eta(r_0)>0$, then $\Ga$ is proper and $r(\Ga)\geq r_0/2$. for given $1\leq j\leq n$, let $X_j:=\{x_i\in V;\ (i,j)\in E\}$ and define 
\[S_{X_j} := \{x\in\R^d;\ |x-x_i|^2=t_{ij},\ \textit{for all}\ \  x_i\in X_j\}.\]
As explained in Section 6, $S_{X_j}$ is a sphere of dimension $d-|X_j|\geq 1$ with radius $r(X_j)\geq r_0/2$. 
Let $\si_{X_j}$ denote the surface area measure on $S_{X_j}$ and write $\nu_{X_j} := \phi_j\,\si_{X_j}$ where $\phi_j$ is a smooth cut-off function supported in an $\eta$-neighborhood of $v_j$ with $\phi_j(v_j)=1$.\\ 

Write $\bx=(x_1,\ldots,x_n)$, $\phi(\bx):=\prod_{j=1}^n \phi_j (x_j)$, then by \eqref{A7} and \eqref{A8}, one has 
\eq\label{4.1}
\int g(\bx)\,\phi (\bx)\,d\om_\F(\bx) = c_j(\Ga_0) \int \int g(\bx)\, \phi(\bx')\,d\nu_{X_j}(x_j) \,d\om_{\F_j} (\bx'),
\ee
where $\bx'= (x_1,\ldots,x_{j-1},x_{j+1},\ldots,x_n)$ and $\F_j=\{f_{il};\ (i,l)\in E, l\neq j\}$. The constant $c_j(\Ga_0)>0$  is the reciprocal of volume of the parallelotope with sides $x_j-x_i$, $(i,j)\in E$  which is easily shown to be at least $c_k r_0^k$, as the distance of each vertex to the opposite face is at least $r_0/2$ on the support of $\phi$.
\bigskip




\section{Proof of Theorem \ref{Thm2.2}.}

Let $d>k$ and again, without loss of generality, assume that $d(\Ga)=1$ and hence $\de(\Ga)=r(\Ga)$. Given $\la,\eps >0$ define the multi-linear expression,
\begin{align}\label{5.1}
&T_{\la\Ga_0}(\mu_\eps)\ :=\\
& \int\cdots\int \mu_\eps(x) \mu_\eps(x-\la x_1)\cdots \mu_\eps(x-\la x_n)\,\phi(x_1,\ldots,x_n) d\om_\F(x_1,\ldots,x_n)\,dx.\nonumber
\end{align}

Given a proper distance graph $\Ga_0=(V,E)$ on $|V|=n$ vertices of degree $k<n$ one has the following upper bound;

\begin{lemma}\label{L4.2} There exists a constant $C=C_{n,d,k}(r_0)>0$ such that
\eq\label{5.2}
|T_{\la\Ga_0} (\mu_{2\eps}) - T_{\la\Ga_0} (\mu_\eps)| \leq\, C\, \la^{-1/2}\, \eps^{(n+\frac{1}{2})(s-d)+\frac{1}{4}}.
\ee
\end{lemma}

This implies again that in dimensions $d-\frac{1}{4n+2}\leq s\leq d$, there exists the limit $\,T_{\la\Ga_0}(\mu):=\lim_{\eps\to 0} T_{\la\Ga_0} (\mu_\eps)$. Also, the lower bound \eqref{3.6} holds for distance graphs of degree $k$, as it was shown for a large class of graphs, the so-called $k$-degenerate distance graphs, see \cite{LM18graph}. Thus one may argue exactly as in Section 3, to prove that there exists a $\la>0$ for which
\eq\label{5.3}
T_{\la\Ga_0}(\mu) >0,\ee
and Theorem 2 follows from the compactness of the configuration space $S_{\la\Ga_0}\subs \R^{dn}$. It remains to prove Lemma \ref{L4.2}.\\

\emph{Proof of Lemma \ref{L4.2}.} Write $\De T(\mu_\eps):=T_{\la \Ga_0}(\mu_\eps)-T_{\la \Ga_0}(\mu_{2\eps})$. Then we have $\De T(\mu_\eps)=\sum_{j=1}\De_j T(\mu_\eps)$, where $\De_j T(\mu_\eps)$ is given by \eqref{5.1} with $\mu_\eps(x-\la x_j)$ replaced by $\De\mu_\eps(x-\la x_j)$ given in \eqref{3.6}, and $\mu_\eps(x-\la x_i)$ by $\mu_{2\eps}(x-\la x_j)$ for $i>j$. Then by \eqref{4.1} we have the analogue of estimate \eqref{3.7}
\eq\label{5.2}|\De T (\mu_\eps)| \ls\,\eps^{(n-1)(s-d)}\,
\int \left|\int \mu_\eps(x)\ \De\mu_\eps\ast_\la \nu_{X_j}(x)\,dx\right|\,\phi(\bx')\,d\om_{\F_j}(\bx'),
\ee
where $\phi(\bx')=\prod_{i\neq j}\phi(x_j)$. Thus by Cauchy-Schwarz and Plancherel, 
\[|\De_j T^\eps (\mu)|^2 \ls\,\eps^{2n(s-d)}\,\int |\widehat{\De_\eps\mu}(\xi)|^2\,I_\la^j (\xi)\,d\xi,\]
where
\[I_\la^j (\xi) = \int |\hat{\nu}_{X_j}(\la\xi)|^2\,\phi(\bx')\,d\om_{\F_j}(\bx').
\]
Recall that on the support of $\phi(\bx')$ $S_{X_j}$ is a sphere of dimension at least 1 and of radius $r\geq r_0/2>0$, contained in an affine subspace orthogonal to $Span\,X_j$.  Thus,
\[|\hat{\nu}_{X_j}(\la\xi)|^2 \ls (1+r_0\la\,dist (\xi,Span\,X_j))^{-1}.\]

Let $U:\R^d\to\R^d$ be a rotation and for $\bx'=(x_i)_{i\neq j}$ write $U\bx'=(Ux_i)_{i\neq j}$. As explained in Section 6, the measure $\om_{\F_j}$ is invariant under the transformation $\bx'\to U\bx'$, hence

\begin{align*}
I_\la (\xi) &\ls \int \int (1+r_0\la\ dist (\xi,Span\,UX_j))^{-1}\,d\om_{\F_j} (\bx')\,dU\\
&= \int \int (1+r_0\la\,|\xi|\ dist(\eta,Span\,X_j))^{-1}\,d\si_{d-1}(\eta)\,d\om{\F_j} (\bx')\\
&\ls (1+r_0\, \la\, |\xi|)^{-1},
\end{align*}
where we have written again $\eta := |\xi|^{-1} U\xi\in S^{d-1}$.


Then we argue as in Lemma 2, noting that  $\widehat{\De\mu_\eps}(\xi)$ is essentially supported on $|\xi|\approx \eps^{-1}$ we have that 
\[|\De T (\mu_\eps)|^2 \ls\,r_0^{-1}\la^{-1}\eps^{2n(s-d)+\frac{1}{2}}\,\int |\hat{\mu}(\xi)|^2 \hat{\phi}(\eps\xi)\,d\xi\ls\,r_0^{-1}\la^{-1}\eps^{(2n+1)(s-d)+\frac{1}{2}},\]  
with $\tilde{\mu}_\eps=\mu_\eps$ or $\tilde{\mu}_\eps =\mu\ast \phi_\eps$. This proves Lemma \ref{L4.2}. $\Box$

\bigskip

\section{Measures on real algebraic sets.}

Let $\F=\{f_1,\ldots,f_n\}$ be a family of polynomials $f_i:\R^d\to\R$. We will describe certain measures supported on the algebraic set 
\eq\label{A1} S_{\F}:= \{x\in\R^d:\ f_1(x)=\ldots=f_n(x)=0\}.\ee

A point $x\in S_\F$ is called \emph{non-singular} if the gradient vectors 
$$ \nabla f_1(x),\ldots,\nabla f_n(x)\ $$ are linearly independent, and let $S_\F^0$ denote the set of non-singular points. It is well-known  and is easy to see, that if $S_\F^0\neq\emptyset$ then it is a relative open, dense subset of $S_\F$, and moreover it is an $d-n$-dimensional sub-manifold of $\R^d$. If $x\in S_\F^0$ then there exists a set of coordinates, $J=\{j_1,\ldots,j_n\}$, with $1\leq j_1<\ldots<j_n\leq d$, such that 

\eq\label{A2} j_{\F,J}(x):= det\, \left( \frac{\partial f_i}{\partial x_j} (x)\right)_{1\leq i\leq n,j\in J}\ \neq 0.\ee

Accordingly, we will call a set of coordinates $J$ \emph{admissible}, if \eqref{A2} holds for at least one point $x\in S_\F^0$, and will denote by $S_{\F,J}$ the set of such points. For a given set of coordinates $x_J$ let 
$\ \nabla_{x_J}f(x):=(\partial_{x_j}f(x))_{j\in J}\ $ and note that $J$ is admissible if and only if the gradient vectors 
$$ \nabla_{x_J}f_1(x),\ldots,\nabla_{x_J}f_n(x)\ $$ are linearly independent at at least one point $x\in S_\F$. It is clear that, unless $S_{\F,J}=\emptyset$, it is a relative open and dense subset of $S_\F$ and is a also $d-n$-dimensional sub-manifold, moreover $S_\F^0$ is the union of the sets $S_{\F,J}$ for all admissible $J$.

We define a measure, near a point $x_0\in S_{\F,J}$ as follows. For simplicity of notation assume that $J=\{1,\ldots,n\}$ and let 
$$\Phi(x):=(f_1,\ldots,f_n,x_{n+1},\ldots,x_d).$$ Then $\Phi:U\to V$ is a diffeomorphism on some open set $x_0\in U\subs \R^d$ to its image $V=\Phi(U)$, moreover $S_\F=\Phi^{-1}(V\cap \R^{d-n})$. Indeed, $x\in S_\F\cap U$ if and only if $\Phi(x)=(0,\dots,0,x_{n+1},\ldots, x_d)\in V$. Let $I=\{n+1,\ldots,d\}$ and write $x_I:=(x_{n+1},\ldots,x_d)$. Let $\Psi(x_I)=\Phi^{-1}(0,x_I)$ and in local coordinates $x_I$ define the measure $\omega_{\F}$ via

\eq\label{A3}\int g\,d\omega_{\F} := \int g(\Psi(x_I))\ Jac^{-1}_\Phi (\Psi(x_I))\,dx_I,
\ee
for a continuous function $g$ supported on $U$. Note that $Jac_\Phi(x)=j_{\F,J}(x)$, i.e. the Jacobian of the mapping $\Phi$ at $x\in U$ is equal to the expression given in \eqref{A2}, and that the measure $d\omega_\F$ is supported on $S_\F$. Define the local coordinates  $y_j=f_j(x)$ for $1\leq j\leq n$ and $y_j=x_j$ for $n<j\leq d$. Then
\[dy_1\wedge\ldots\wedge dy_d = df_1\wedge\ldots\wedge df_n\wedge dx_{n+1}\wedge\ldots \wedge dx_d = Jac_\Phi(x)\,dx_1\wedge\ldots\wedge dx_d,\]
thus
\[dx_1\wedge\ldots\wedge dx_d = Jac_\Phi(x)^{-1} df_1\wedge\ldots\wedge df_n \wedge dx_{n+1}\wedge\ldots \wedge dx_d = df_1\wedge\ldots\wedge df_n \wedge d\omega_\F.\]

This shows that the measure $d\omega_\F$ (given as a differential $d-n$-form on $S_{\F}\cap U$) is independent of the choice of local coordinates $x_I$. Then $\omega_\F$ is defined on $S_\F^0$ and moreover the set $S_\F^0\backslash S_{\F,J}$ is of measure zero with respect to $\omega_F$, as it is a proper analytic subset on $\R^{d-n}$ in any other admissible local coordinates.\\


Let $x=(z,y)$ be a partition of coordinates in $\R^d$, with $y=x_{J_2}$, $z=X_{J_1}$, and assume that for $i=1,\ldots,m$ the functions $f_i$ depend only on the $z$-variables. We say that the partition of coordinates is \emph{admissible}, if there is a point $x=(z,y)\in S_\F$ such that both the gradient vectors $\nabla_z f_1(x),\ldots,\nabla_z f_m(x)$ and the vectors $\nabla_y f_{m+1}(x),\ldots,\nabla_y f_n(x)$ for a linearly independent system. Partition the system $\F=\F_1\cup\F_2$ with $\F_1=\{f_1,\ldots,f_m\}$ and $\F_2=\{f_{m+1},\ldots,f_n\}$. Then there is set $J_1'\subs J_1$ for which 
\[j_{\F_1,J_1'}(z):= \det \left(\frac{\partial f_i}{\partial x_j}(z)\right )_{1\leq i\leq m, \,j\in J_1'}\neq 0,\]
and also a set $J_2'\subs J_2$ such that 
\[j_{\F_2,J_2'}(z,y):= \det \left(\frac{\partial f_i}{\partial x_j}(z,y)\right)_{m+1\leq i\leq n,\, j\in J_2'}\neq 0.\]
Since $\nabla_y f_i\equiv 0$ for $1\leq i\leq m$, it follows that the set of coordinates $J'=J_1'\cup J_2'$ is admissible, moreover
\[j_{\F,J'}(y,z) = j_{\F_1,J_1'}(z)\,j_{\F_2,J_2'}(y,z).\]

For fixed $z$, let $f_{i,z}(y):=f_i(z,y)$ and let $\F_{2,z}=\{f_{m+1,z},\ldots,f_{n,z}\}$. Then clearly $j_{\F_2,J_2'}(y,z)=j_{\F_{2,z},J_2'}(y)$ as it only involves partial derivatives with respect to the $y$-variables. Thus we have an analogue of Fubini's theorem, namely

\eq\label{A7}
\int g(x)\,d\omega_\F (x) \,=\, \int\int g(z,y)\,d\omega_{\F_{2,z}}(y)\, d\omega_{\F_1}(z).\ee

\bigskip

Consider now algebraic sets given as the intersection of spheres. Let $x_1,\ldots,x_m\in\R^d$, $t_1,\ldots,t_m>0$ and $\F=\{f_1,\ldots,f_m\}$ where $f_i(x)=|x-x_i|^2-t_i$ for $i=1,\ldots,m$. Then $S_\F$ is the intersection of spheres centered at the points $x_i$ of radius $r_i=t_i^{1/2}$. If the set of points $X=\{x_1,\ldots,x_m\}$ is in general position (i.e they span an $m-1$-dimensional affine subspace), then a point $x\in S_\F$ is non-singular if
$x\notin span\,X$, i.e if $x$ cannot be written as linear combination of $x_1,\ldots,x_m$. Indeed, since $\nabla f_i(x)=2(x-x_i)$ we have that
\[\sum_{i=1}^m a_i\nabla f_i(x)=0\ \Longleftrightarrow\ \sum_{i=1}^m a_i\,x = \sum_{i=1}^m a_i x_i,\]
which implies $\sum_{i=1}^m a_i=0$ and $\sum_{i=1}^m a_i x_i=0$. By replacing the equations $|x-x_i|^2=t_i$ with $|x-x_1|^2-|x-x_i|^2=t_1-t_i$, which is of the form $x\cdot (x_1-x_i)=c_i$, for $i=2,\ldots,m$, it follows that $S_\F$ is the intersection of sphere with an $n-1$-codimensional affine subspace $Y$, perpendicular to the affine subspace spanned by the points $x_i$. Thus $S_\F$ is an $m$-codimensional sphere of $\R^d$ if $S_\F$ has one point $x\notin span \{x_1,\ldots,x_m\}$ and all of its points are non-singular. Let $x'$ be the orthogonal projection of $x$ to $span X$. If $y\in Y$ is a point with $|y-x'|=|x-x'|$ then by the Pythagorean theorem we have that $|y-x_i|=|x-x_i|$ and hence $y\in S_\F$. It follows that $S_\F$ is a sphere centered at $x'$ and contained in $Y$.\\

Let $T=T_X$ be the inner product matrix with entries $t_{ij}:=(x-x_i)\cdot (x-x_j)$ for $x\in S_\F$. Since $(x-x_i)\cdot (x-x_j)=1/2( t_i+t_j - |x_i-x_j|^2)$ the matrix $T$ is independent of $x$. We will show that $d\om_\F=c_T\, d\si_{S_\F}$ where $d\si_{S_\F}$ denotes the surface area measure on the sphere $S_\F$ and $c_T= 2^{-m} det (T)^{-1/2}>0$, i.e for a function $g\in C_0(\R^d)$,
\eq\label{A8}
\int_{S_\F} g(x)\,d\om_\F(x) = c_T\int_{S_\F} g(x)\,d\si_{S_\F}(x).
\ee

Let  $x\in S_\F$ be fixed and let $e_1,\ldots,e_d$ be an orthonormal basis so that the tangent space $\ T_x S_\F= Span \{e_{m+1},\ldots,e_d\}$ and moreover we have that\\  $\ Span \{\nabla f_1,\ldots,\nabla f_m\}= Span\{e_1,\ldots,e_m\}\,$. Let $x_1,\ldots,x_n$ be the corresponding coordinates on $\R^d$ and note that in these coordinates the surface area measure, as a $d-m$-form at $x$, is 
\[d\si_{S_\F}(x)=d x_{m+1}\wedge\ldots\wedge d x_d.\]
On the other hand, in local coordinates $x_I=(x_{m+1},\ldots,x_d)$, it is easy to see form \eqref{A2}-\eqref{A3} that $j_{\F,J}(x)= 2^m\,vol (x-x_1,\ldots,x-x_m)$ and hence
\[d\om_\F(x) = 2^{-m} vol (x-x_1,\ldots,x-x_m)^{-1}\,d x_{m+1}\wedge\ldots\wedge d x_d,\]
where $vol (x-x_1,\ldots,x-x_m)$ is the volume of the parallelotope with side vectors $x-x_j$. Finally, it is a well-known fact from linear algebra that 
\[vol (x-x_1,\ldots,x-x_m)^2 = det\, (T),\] 
i.e. the volume of a parallelotope is the square root of the Gram matrix formed by the inner products of its side vectors.

\vspace{0.5in}


\end{document}